\documentclass[reqno, 14]{amsart}
\usepackage{amsmath}
\usepackage{amscd,amsthm,amsfonts,amsopn,amssymb,verbatim}
\parskip =\medskipamount
\numberwithin{equation}{section}
\newtheorem{theorem}{Theorem} %[section]

\newtheorem{lemma}[theorem]{Lemma}

\theoremstyle{remark}

\def\be{\begin{equation}}
\def\ee{\end{equation}}

\def\ve{\varepsilon}

\allowdisplaybreaks

\begin{document}
\Large
\title{On a homogenization problem}
\date{\today}
\author{J.~Bourgain}
\address{J.~Bourgain, Institute for Advanced Study, Princeton, NJ 08540}
\email{bourgain@math.ias.edu}
\thanks{The author was partially supported by NSF grants DMS-1301619}
\begin{abstract}
We study a homogenization question for a stochastic divergence type operator.
\end{abstract}

\maketitle

\section
{Introduction and Statement}

Let $\{\sigma_x(\omega); x\in\mathbb Z^d\}$ be i.i.d., $\mathbb E[\sigma_x]=0$ and assume moreover,
$$
\Vert \sigma_x\Vert_\infty \leq C.
\eqno{(1.1)}
$$
Consider the random operator
$$
L_\omega =-\Delta+\delta \nabla^*\sigma \nabla
\eqno{(1.2)}
$$
where $\nabla f(x) =f(x+1)-f(x)$ for $f$ on $\mathbb Z^d$.

Consider the stochastic equation
$$
L_\omega u_\omega =f.
\eqno{(1.3)}
$$
Formally we have
$$
\langle u_\omega\rangle= \langle L_\omega^{-1}\rangle f \ \text { and } \ A \langle u_\omega\rangle =f
\eqno{(1.4)}
$$
with
$$
A=\langle L_\omega^{-1} \rangle ^{-1}.
\eqno{(1.5)}
$$
We prove the following

\noindent
{\bf Theorem.}
With the above notation, given $\ve>0$, there is $\delta_0>0$ such that for $|\delta| <\delta_0$, $A$ has the form
$$
A =\nabla^* (1+K_1) \nabla
\eqno{(1.6)}
$$
with $K_1$ given by a convolution operator with symbol satisfying
$$
\hat K_1 (\xi)= O(\delta|\xi|^{d-\ve}).
\eqno{(1.7)}
$$
\newpage

\noindent
{\bf Remarks}

\begin{enumerate}
\item
This paper is closely related to a note of M.~Sigal \cite {S}, where the exact same problem is considered.
In \cite {S} an asymptotic expansion for $K_1$ is given and (1.7) verified up to the leading order.
What we basically manage to do here is to control the full series.
The argument is rather simple, but contains perhaps some novel ideas that may be of independent interest in the 
study of the averaged dynamics of stochastic PDE's.
\medskip

\item The author is grateful to T.~Spencer for bringing the problem to his attention and a few preliminary discussions.
\end{enumerate}

\section {The Expansion}

We briefly recall the derivation of the multi-linear expansion for $K_1$ established in \cite{S}.
Denote $b=\delta\sigma, P=\mathbb E, P^\bot = 1-\mathbb E$.
Using the Feshback-Shur map to the block decomposition
$$
\begin{pmatrix} (P, P) & (P^\bot, P)\\  (P, P^\bot) & (P^\bot, P^\bot)\end{pmatrix}
$$
we obtain
$$
PL^{-1}P= \big(PLP-PLP^\bot (P^\bot LP^\bot -io)^{-1} P^\bot LP \big)^{-1}
$$
Since $PLP=\Delta P, PLP^\bot =P\nabla^*b\nabla P^\bot, P^\bot LP= P^\bot \nabla^* b\nabla P$, we obtain
$$
(-\Delta P -\nabla^* Pb\nabla (P^\bot LP^\bot)^{-1} \nabla^*b  \nabla P)^{-1}.
\eqno{(2.1)}
$$
Next, $P^\bot LP^\bot = (-\Delta)\big( 1+ (-\Delta)^{-1} \nabla^*P^\bot b\nabla\big)$ and we expand
$$
(P^\bot LP^\bot)^{-1} =\Big[1-(-\Delta)^{-1} \nabla^* P^\bot b\Big(\sum_{n\geq 0} (-1)^n(KP^\bot b)^n\Big) \nabla P^\bot\Big]
(-\Delta)^{-1}
\eqno{(2.2)}
$$
where we denoted $K$ the convolution singular operator
$$
K=\nabla (-\Delta)^{-1} \nabla^*.
\eqno{(2.3)}
$$
Substitution of (2.2) in (2.1) gives
$$
\langle \nabla^* b\nabla (P^\bot LP^\bot)^{-1}  \nabla^* b\nabla\rangle =\sum_{n\geq 1}(-1)^{n+1} \nabla^* \langle b(KP^\bot b)^n\rangle \nabla.
\eqno{(2.4)}
$$
Hence
$$
\langle L_\omega^{-1}\rangle =\big(-\Delta + (2.4)\big)^{-1}
$$
and
$$
A=-\Delta + (2.4) =\nabla^* \big(1+K_1)\nabla
$$
with
$$
K_1=\sum_{n\geq 1} (-1)^{n+1} \langle b(KP^\bot b)^n\rangle.
\eqno{(2.5)}
$$
Remains to analyze the individual terms in (2.5).

\section
{A deterministic inequality}

Our first ingredient in controlling the multi-linear terms in the series (2.5) is the following (deterministic) bound on composing singular 
integral and multiplication operators.

\begin{lemma}
Let $K$ be a (convolution) singular integral operator acting on $\mathbb Z^d$ and $\sigma_1, \ldots, \sigma_s \in \ell^\infty (\mathbb Z^d)$.
Define the operator
$$
T = K\sigma_1 K\sigma_2 \cdots K\sigma_s.
\eqno{(3.1)}
$$
Then $T$ satisfies the pointwise bound
$$
|T(x_0, x_s)|<|x_0-x_s|^{-d+\ve} (C\ve^{-1})^s \prod^s_1 \Vert\sigma_j\Vert_\infty
\eqno{(3.2)}
$$
for all $\ve>0$.
\end{lemma}

\begin{proof}
Firstly, recalling the well-known bound
$$
\Vert K\Vert_{p\to p}< \frac c{p-1} \ \text { for } \ 1< p\leq 2
\eqno{(3.3)}
$$
and normalizing $\Vert\sigma_j\Vert_\infty =1$, we get
$$
\Vert T\Vert_{p\to p} +\Vert T^*\Vert_{p\to p}< \Big(\frac c{p-1}\Big)^s.
\eqno{(3.4)}
$$
In particular
$$
\max_x \Big(\sum_y |T(x, y)|^p\Big)^{\frac 1p}+\max_y \Big(\sum_x|T(x, y)|^p\Big)^{\frac 1p} <\Big(\frac c{p-1}\Big)^s.
\eqno{(3.5)}
$$
Next, write
$$
T_s (x_0, x_s)=\sum_{x_1, \ldots, x_{s-1}} K(x_0-x_1)\sigma_1(x_1) K(x_1-x_2)\sigma_2(x_2) \cdots K(x_{s-1} -x_s)
\sigma_s (x_s).
\eqno{(3.6)}
$$

Specify $R\gg 1$ and $0\leq i<s$ satisfying
$$
|x_i -x_{i+1}| \sim R
\eqno{(3.7)}
$$
$$
\max_j|x_j -x_{j+1}|\lesssim R.
\eqno{(3.8)}
$$
In particular $|x_0 -x_s|\lesssim sR$.
The corresponding contribution to (3.6) may be bounded by
$$
\sum_{\substack {x_i, x_{i+1}\\ |x_i -x_{i+1}|\sim R}}
|T^{(*)}_{i} (x_0, x_i)| \ |K(x_i-x_{i+1})| \ | T^{(*)}_{s-1-i}(x_{i+1}-x_s)| 
\eqno{(3.9)}
$$
with $T^{(*)}_i$ obtained from formula (3.6) with additional restriction (3.8).
The bound (3.5) also holds for $T^{(*)}_i$.
Since $|K(z)|< |z|^{-d}$ (where we denote $| \ |=| \ |+1$), it follows from (3.5), (3.7), (3.8) and 
H\"older's inequality that
$$
\begin{aligned}
(3.9) &\leq\Big(\frac c{p-1}\Big)^s \Big(\sum_{x_i, x_{i+1}, |x_0-x_i|< sR, |x_0-x_{i+1}| <s R} 1\Big)^{1/p'} R^{-d} \quad \Big(p'=\frac p{p-1}\Big)\\[12pt]
&< \Big(\frac c{p-1}\Big)^s (sR)^{2d(p-1)} R^{-d}.
\end{aligned}
\eqno{(3.10)}
$$
Remains to take $p$ such that $2d(p-1)=\ve$.
Then
$$
(3.10) < \Big(\frac c{p-1}\Big)^s R^{-d+\ve}< s^d \big(\frac c{p-1}\Big)^s |x_0- x_s|^{-d+\ve}
$$
proving (3.2).
\end{proof}

\section
{Use of the randomness}
Returning to (2.5), the randomness will allow us to further improve the pointwise bounds on $\langle b(KP^\bot b )^n\rangle$.

Write
$$
b(Kp^\bot b)^n (x_0, x_n)=\sum_{x_1, \ldots, x_{n-1}\in\mathbb Z^d} b_{x_0} K(x_0, x_1)P^\bot b(x_1) K(x_1, x_2)P^\bot b(x_2)\ldots b(x_n).
\eqno{(4.1)}
$$
Note that evaluation of $\langle b(KP^\bot b)^n\rangle$ by summation over all diagrams would produce combinatorial factors growing more
rapidly than $C^n$ and hence we need to proceed differently.

Let again $R\gg 1$ and $0\leq j_0<n$ s.t.
$$
|x_{j_0} - x_{j_0+1}|\sim R \ \text { and } \ \max_{0\leq j< n} |x_j-x_{j+1}|\lesssim R.
\eqno{(4.2)}
$$
We denote
$$
S=\Big\{\begin{matrix}&(x_1, \ldots, x_{n-1})\in (\mathbb Z^d)^{n-1} \ \text { and } \ \{x_0, \ldots, x_{j_0}\}\cap \{x_{j_0+1}, \ldots, x_n\}\not= \phi\\
&\text {\rm  subject to (4.2)} \hfill \end{matrix} \Big\}
\eqno{(4.3)}
$$
Using the lace expansion terminology, $\mathbb E[(4.1)]$ only involves the irreducible graphs in (4.1), due to the presence of the projection
operators $P^\bot$ (this is the only place where we refer to the lace expansion which by itself seems inadequate to evaluate $\mathbb E[(4.1)]$
because of the role of cancellations).
From the preceding, it follows in particular that
$$
\mathbb E[(4.1)]=\mathbb E[(4.4)]
$$
defining
$$
(4.4) =\sum_{(x_1, \ldots, x_{n-1})\in S} b_{x_0} K(x_0, x_1) P^\bot b_{x_1} \ldots b_{x_n}.
$$
Our goal is to prove

\begin{lemma}
For all $\ve>0$, we have
$$
|\mathbb E[(4.4)]|<C^n_\ve |x_0-x_n|^{-2d+\ve}\eqno{(4.5)}
$$
\end{lemma}
which clearly implies the Theorem.

For definition (4.3)
$$
S= \bigcup_{\substack {0\leq j_1\leq j_0\\ j_0<j_2\leq n}} S_{j_1, j_2}
$$
where
$$
S_{j_1, j_2}= \{(x_1,\ldots, x_{n-1})\in (\mathbb Z^d)^{n-1} \text { subject to (4.2) and } x_{j_1}=x_{j_2}\}.
\eqno{(4.6)}
$$
Note that these sets $S_{j_1, j_2}$ are not disjoint and we will show later how to make them disjoint at the cost
of another factor $C^n$.

Consider the sum
$$
\sum_{(x_1, \ldots, x_{n-1})\in S_{j_1, j_2}} b_{x_0} K(x_0, x_1) P^\bot b_{x_1} \cdots b_{x_n}= (4.7).
$$
We claim that for all $\ve>0$
$$
|(4.7)|< C_\ve^n R^{-d+4\ve} |x_0-x_n|^{-d}
\eqno{(4.8)}
$$
(thus without taking expectation).

To prove (4.8), factor (4.7) as
$$
\begin{aligned}
&(KP^\bot b)^{j_1}(x_0, x_{j_1})(KP^\bot b)^{j_0-j_1}(x_{j_1}, x_{j_0}) K(x_{j_0}, x_{j_0+1}) P^\bot b_{x_{j_0+1}},\\[4pt]
& (KP^\bot b)^{j_2-j_0} (x_{j_0+1}, x_{j_1}) (KP^\bot b)(x_{j_1}, x_n)
\end{aligned}
\eqno{(4.9)}
$$
with summation over $x_{j_0}, x_{j_0+1}, x_{j_1}$.

Using the deterministic bound implied by Lemma 1
$$
|(KP^\bot b)^\ell (x, y)|< C_\ve^\ell |x-y|^{-d+\ve}
\eqno{(4.10)}
$$
we may indeed estimate
$$
\begin {aligned}
|(4.7)| &< R^{-d} C_\ve^n \sum_{x_{j_0}, x_{j_0+1}, x_{j_1}} |x_0-x_{j_1}|^{-d+\ve} |x_{j_1}-x_{j_0}|^{-d+\ve} |x_{j_0+1} -x_{j_1}|^{-d+\ve}
|x_{j_1}-x_n|^{-d+\ve}<\\[12 pt]
&< C_\ve^n R^{-d+4\ve} |x_0-x_n|^{-d}.
\end{aligned}
$$
Remains the disjointification issue for the sets $S_{j_1, j_2}$.

Our devise to achieve this may have an independent interest.  Define the disjoint sets
$$
S_{j_1, j_2}' =S_{j_1, j_2} \Big\backslash  \ \Big(\bigcup_{\substack{j<j_1\\ j_0< j'\leq n}} S_{j, j'} \ \cup \ \bigcup_{j_0< j'<j_2} S_{j_1, j'}\Big).
\eqno{(4.11)} 
$$ 
Replacing $S_{j_1, j_2}$ by $S_{j_1, j_2}'$ in (4.7), we prove that the bound (4.8) is still valid.

Note that, by definition, $(x_1, \ldots, x_{n-1})\not\in \operatornamewithlimits\bigcup\limits_{\substack{j<j_1\\ j_0<j'\leq n}} S_{j. j'}$ means that
$$
\{x_0, \ldots, x_{j_1-1}\}  \cap \{x_{j_0+1}, \dots, x_n\}=\phi.
\eqno{(4.12)}
$$
Thus we need to implement the condition (4.12) in the summation (4.7) at the cost of a factor bounded by $C^n$.

We introduce an additional set of variables $\bar\theta =(\theta_x)_{x\in\mathbb Z^d}, \theta_x\in \mathbb T=\mathbb R/\mathbb Z$
and consider the corresponding Steinhaus system.
Denote $E=\{0, 1, \ldots, j_1-1\}$, $F=\{j_{0}+ 1, \ldots, n\}$.
Replace in (4.7)
$$
\begin{cases} b_{x_j} \ \text { by } \ b_{x_j} e^{i\theta_{x_j}} \ \text { for } \ j\in E\\
b_{x_j} \ \text { by } \ b_{x_j} e^{-i\theta_{x_j}} \ \text { for } \ j\in F.
\end{cases}
\eqno{(4.13)}
$$
After this replacement, (4.7) becomes a Steinhaus polynomial in $\bar\theta$, i.e. we obtain
$$
\sum_{(x_1, \ldots, x_{n-1})\in S_{j_1, j_2}}  \ e^{i(\sum_{j\in E} \theta_{x_j}-\sum_{k\in F} \theta_{x_k})}
b_{x_0} K(x_0, x_1) P^\bot b_{x_1}\ldots b_{x_n}= (4.14)
$$
for which the estimate (4.8) still holds (uniformly in $\bar\theta$).

Next, performing a Poisson convolution in each $\theta_x$ (which is a contraction), gives
$$
\begin{aligned}
&\int(4.14) \prod_x P_t(\theta_x'-\theta_x) d\theta_x=\\
&\sum_{(x_1, \ldots, x_{n-1})\in S_{j_1, j_2}} t^{w_{\bar x}} e^{i(\sum_{j\in E}\theta_{x_j} -\sum_{k\in F} \theta_{x_k})}
b_{x_0} K(x_0, x_1)P^\bot\cdots b_{x_n}
\end{aligned}
\eqno{(4.15)}
$$
where $0\leq t\leq 1$ and
$$
w_{\bar x} =\sum_x \big| \, |\{j\in E; x_j=x\}| - |\{k\in F; x_k =x\}|\big| \leq |E|+|F|=D.
$$
Note that the condition $\{x_j, j\in E\}\cap \{x_k; k\in F\}=\phi$ is equivalent to $w_{\bar x}=D$ and (4.14) obtained by projection of (4.15),
viewed as polynomial $t$, on the top degree term.
Our argument is then concluded by the standard Markov brothers' inequality.

\begin{lemma}
Let $P(t)$ be a polynomial of degree $\leq D$.
Then
$$
\max_{-1\leq t\leq 1} |P^{(k)} (t)|\leq\frac {D^2(D^2-1^2)(D^2-2^2)\cdots (D^2-(k-1)^2)}{1, 3, 5\ldots (2D-1)} \max_{-1\leq t\leq 1}|P(t)|.
\eqno{(4.16)}
$$
\end{lemma}

Indeed, we conclude that for all $\bar\theta$
$$
\Big|\sum_{\substack{(x_1, \ldots, x_{n-1})\in S_{j_1, j_2}\\ w_{\bar x}=D}} e^{i(\sum_{j\in E}\theta_{x_j} -\sum_{k\in F}\theta_{x_k})} b_{x_0} K(x_0, x_1)P^\bot \ldots b_{x_n}\Big|
<C^n.(4.8)
$$
and set then $\bar\theta=0$.


\begin{thebibliography}
{XXXXX}
\bibitem [S]{S}  M.~Sigal, \emph {Homogenization problem}, preprint.
\end{thebibliography}
\end{document}